\documentclass[a4paper,11pt]{amsart}
\usepackage[utf8]{inputenc}
\usepackage{amsaddr}
\usepackage{calrsfs}
\usepackage{graphicx,color}
\usepackage{amssymb}
\usepackage{xcolor}
\usepackage[colorlinks]{hyperref}
\usepackage[margin=3cm]{geometry}

\numberwithin{equation}{section}

\theoremstyle{definition}
\newtheorem{dfn}{Definition}[section]

\theoremstyle{plain}
\newtheorem{thm}{Theorem}[section]
\newtheorem{pro}{Proposition}[section]

\theoremstyle{definition}
\newtheorem{rem}{Remark}[section]
\newtheorem{exa}{Example}[section]

\newcommand{\N}{\mathbb{N}}
\newcommand{\R}{\mathbb{R}}
\newcommand{\E}{\mathbb{E}}
\renewcommand{\P}{\mathbb{P}}

\newcommand{\Cov}{\mathbb{C}\mathrm{ov}}
\newcommand{\F}{\mathcal{F}}
\newcommand{\K}{\mathrm{K}}
\newcommand{\e}{\mathrm{e}}
\newcommand{\la}{\langle}
\newcommand{\ra}{\rangle}

\newcommand{\1}{\mathbf{1}}
\renewcommand{\d}{\mathrm{d}}

\begin{document}
\title[Gaussian Volterra Processes with Jump]
{Prediction of Gaussian Volterra Processes with Compound Poisson Jumps}
\date{\today}

\author[Maleki Almani]{Hamidreza Maleki Almani}
\address{Department of Mathematics and Statistics, University of Vaasa, P.O. Box 700, FIN-65101 Vaasa, FINLAND}
\email{hamidreza.maleki@uwasa.fi}

\author[Shokrollahi]{Foad Shokrollahi}
\address{Department of Mathematics and Statistics, University of Vaasa, P.O. Box 700, FIN-65101 Vaasa, FINLAND}
\email{foad.shokrollahi@uwasa.fi}

\author[Sottinen]{Tommi Sottinen}
\address{Department of Mathematics and Statistics, University of Vaasa, P.O. Box 700, FIN-65101 Vaasa, FINLAND}
\email{tommi.sottinen@uwasa.fi}

\begin{abstract}
We consider a Gaussian Volterra process with compound Poisson jumps and derive its prediction law.
\end{abstract}




\maketitle


\section{Introduction}

We are interested in a mixed process $X=G+J$, where the continuous part $G$ is a so-called Gaussian Volterra process.  In older terminology these are processes that admit canonical representation of multiplicity one.  A typical Gaussian Volterra process is the fractional Brownian motion as shown in Norros et al. \cite{Norros-Valkeila-Virtamo-1999}.  See Section \ref{sect:examples} for the definition of fractional Brownian motion and for other examples.  Intuitively, a Gaussian process $G$ is a Gaussian Volterra process if one can construct a martingale $M$ from by using a non-anticipative linear transformation and then represent the original process $G$ in a non-anticipative way as a linear transformation of the martingale $M$.  The motivation to use Gaussian Volterra processes is that for them one can calculate their prediction law in terms of the kernels that transfer the Gaussian Volterra process into its driving martingale $M$ and vice versa.
In the mixture $X=G+J$ the jump part $J$ will be an independent compound Poisson process with square-integrable jump distribution.

One motivation to study processes of the type $X=G+J$ comes from mathematical finance.  Indeed, it is well-known that the returns of financial assets do not follow Gaussian distribution \cite{blattberg2010comparison,fama1965behavior,mandelbrot1997variation} and the returns also exhibit jumps, or shocks \cite{akgiray1986stock,akgiray1987compound,ball1985jumps,jarrow1984jump,press1967compound}.  Also, there is evidence of long-range dependence in the returns also explain the presence of long-memory \cite{baillie1996long,chan2006stock,harvey1995predictable,kim2008sovereign,rajan2003great}.  Thus models where the returns are Gaussian with jumps seems more reasonable: The Gaussian part could take care of the long-range dependence with fractional Brownian motion (fBm) as the Gaussian Volterra process, and the shocks would come from the compound Poisson part.  Then one can use the result of this paper to calculate imperfect hedges in the mixed model in the similar way as done in \cite{Shokrollahi-Sottinen-2017, Sottinen-Viitasaari-2018a}.  Indeed, this is work in progress by the authors.

In this paper we derive the prediction law of the mixed process $X=G+J$.

The rest of the paper is organized as follows.  In Section \ref{sect:gvp} we define Gaussian Volterra processes and derive their prediction laws.  Section \ref{sect:gvp_jumps} is the main section of the paper where we introduce the Gaussian Volterra processes with compound Poisson jumps and derive their prediction laws.  Finally, in Section \ref{sect:examples} we provide examples of Gaussian Volterra processes. 

\section{Gaussian Volterra processes}\label{sect:gvp}

A Gaussian Volterra process is a Gaussian process that has a canonical representation of multiplicity one with respect to a Gaussian martingale.  The Gaussian Volterra process is defined in Definition \ref{dfn:gvp} below in terms of covariance functions. The Gaussian Volterra representations follow from Definition \ref{dfn:gvp} and are stated in Proposition \ref{pro:gvp}. 

For convenience we consider processes over the compact time-interval $[0,T]$ with an arbitrary but fixed time horizon $T>0$.

Let $G=(G_t)_{t\in[0,T]}$ be a centered Gaussian process with $G_0=0$ and covariance function $R\colon[0,T]^2\to\R$ defined on a complete probability space $(\Omega,\F,\P)$.

A kernel $K:[0,T]^2\to\R$ is a Volterra kernel if $K(t,s)=0$ whenever $t<s$.  For a Volterra kernel $K$ we define its associated operator $\K$ as.
$$
\K[f](t) = \int_0^t f(s)K(t,s)\, \d s.
$$
Denote 
\begin{eqnarray*}
\1_t(s) &=& \1_{[0,t)}(s) = \left\{
\begin{array}{rl}
1, & \mbox{if } s\in [0,t), \\
0, & \mbox{otherwise}
\end{array}
\right..
\end{eqnarray*} 
The adjoint associated operator $\K^*$ of the Volterra kernel $K$ is given by extending linearly the relation
$$
\K^*[\1_t](s) = K(t,s). 
$$
It turns out that $\K^*_G$ for a Gaussian Volterra process with covariance
$$
R(t,s) = \int_0^{t\wedge s} K(t,u)K(s,u)\, \d v(u)
$$
extends to an isometry from $\Lambda$ to $L^2([0,T],\d v)$ where $v$ is given in Definition \ref{dfn:gvp}(i) and $\Lambda$, the space of wiener integrands, is the closure of the indicator functions $\1_t$, $t\in[0,T]$, in the inner product
$$
{\la \1_t,\1_s\ra}_{\Lambda} = R(t,s).
$$ 

\begin{rem}
By A\`os et al. \cite{Alos-Mazet-Nualart-2001}, if $K$ is of bounded variation in its first argument, we can write for any elementary $f$ 
$$
\K^*[f](t) = f(t)K(T,t) - \int_t^T\left[f(u)-f(t)\right]K(\d u, t).
$$
Moreover, we have for elementary $f$ and $g$ that
$$
\int_0^T \K^*[f](t)g(t)\, \d t 
=
\int_0^T g(t) \K[h](\d t)
$$
justifying the name ``adjoint'' associated operator.
\end{rem}

For Gaussian Volterra representations we recall what is the co-called abstract Wiener integral (for more information on abstact Wiener integrals and their relation to conditioning we refer to \cite{Sottinen-Yazigi-2014})).

The linear space $\mathcal{L}$ is the closure of of the random variables $G_t$, $t\in[0,T]$, in $L^2(\Omega,\F,\P)$.  The spaces $\Lambda$ and $\mathcal{L}$ are isometric.  Indeed, the mapping
$$
\1_t \mapsto G_t
$$
extends to an isometry.  This isometry is called the abstract Wiener integral and we denote it 
$$
\int_0^T f(t)\, \d G_t
$$
for a $f\in\Lambda$.

\begin{dfn}[Gaussian Volterra process]\label{dfn:gvp}
Let $G=(G_t)_{t\in[0,T]}$ be a centered Gaussian process with covariance function $R\colon [0,\infty)^2 \to \R$. Assume that	
\begin{enumerate}
\item there exists an increasing function $v\colon[0,T]\to\R$ and a Volterra kernel $K\colon[0,T]^2\to\R$ such that $\int_0^t K(t,s)^2 \, \d v(s)<\infty$ for all $t\in [0,T]$ and
$$
R(t,s) = \int_0^{t\wedge s} K(t,u)K(s,u)\, \d v(u),
$$
\item for each $t\in[0,T]$ the equation
$$
\K^*[K^{-1}(t,\cdot)](s) = \1_t(s)
$$	 
admits a solution $K^{-1}(t,\cdot)$.
\end{enumerate}
\end{dfn}

Note that by Definition \ref{dfn:gvp}(ii) the operator $\K^*$ is invertible and we have
$$
(\K^*)^{-1}[\1_t](s) =  K^{-1}(t,s). 
$$

We note that due to Definition \ref{dfn:gvp}(i) for a Gaussian Volterra process the space $\Lambda$ is isometric to $L^2([0,T], \d v)$.  Indeed, we have 
$$
{\la f,g\ra}_{\Lambda} = {\la \K^*[f], \K^*[g]\ra}_{L^2([0,T],\d v)}.
$$
In particular, this means that the mapping $\K^*$ in Definition \ref{dfn:gvp}(ii) is an isometry between $\Lambda$ and $L^2([0,t], \d v)$.  We also note that due to Definition \ref{dfn:gvp}(ii) the operator $\K^*$ is invertible and we have, in particular, that
$$
(\K^*)^{-1}[\1_t](s) = K^{-1}(t,s).
$$

The following representation proposition is a direct consequence of Definition \ref{dfn:gvp}.  Indeed, Proposition \ref{pro:gvp} could have been taken as the definition of Gaussian Volterra process.

\begin{pro}[Volterra representation]\label{pro:gvp}
Let $G$ be a Gaussian Volterra process.  Let $K^{-1}$ be the kernel in Definition \ref{dfn:gvp}(ii).  Then the process 
$$
M_t = \int_0^t K^{-1}(t,s)\, \d G_s
$$
is a Gaussian martingale with bracket $\la M\ra_t = v(t)$.  Moreover, 
$$
G_t = \int_0^t K(t,s)\, \d M_s,
$$
where $v$ and $K$ are as in Definition \ref{dfn:gvp}(i).
\end{pro}	 

Note that form Proposition \ref{pro:gvp} we immediately see that the filtrations $\mathbb{F}^G$ and $\mathbb{F}^M$ coincide.  

The Volterra representations of Proposition \ref{pro:gvp} extend immediately to the following transfer principle for Wiener integrals.

\begin{pro}[Transfer principle]\label{pro:tp}
Let $f\in\Lambda$ and $g\in L^2([0,T],\d v)$. Then
\begin{eqnarray*}
\int_0^T f(t)\, \d G_t &=& \int_0^T \K^*[f](t)\, \d M_t, \\
\int_0^T g(t)\, \d M_t &=& \int_0^T (\K^*)^{-1}[g](t)\, \d G_t.
\end{eqnarray*}	
\end{pro}

In what follows we will use the following notation for the conditional mean, the conditional covariance and the conditional law of a stochastic process $Y$:
\begin{eqnarray*}
\hat m_t^Y(u)      &=& \E\left[Y_t \big| \F^Y_u\right], \\
\hat R_Y(t,s|u)    &=& \Cov[Y_t,Y_s\big| \F^Y_u], \\
\hat P_t^Y(\d y|u) &=& \P\left[Y_t\in \d y\big|\F^Y_u\right]
\end{eqnarray*}

We end this section by stating the prediction formula for Gaussian Volterra processes.  The formula and its proof is similar to that given in \cite{Sottinen-Viitasaari-2017b}.  We give here the proof in detail for the convenience of the readers.

\begin{pro}[Volterra Prediction]\label{pro:pred}
Let $G$ be a Gaussian Volterra process as in Definition \ref{dfn:gvp}. Let $u\le s\le t\le T$. Denote
$$
\Psi(t,s|u) = (\K^*)^{-1}[K(t,\cdot)-K(u,\cdot)](s) 
$$
and
$$
\Phi(\d x ; \mu,\sigma^2) =
\frac{1}{\sqrt{2\pi\sigma^2}}\e^{-\frac12\frac{(x-\mu)^2}{\sigma^2}}\, \d x.
$$
Then
\begin{eqnarray*}
\hat m^G_t(u) &=& G_u - \int_0^u \Psi(t,s|u)\, \d G_u,\\
\hat R_G(t,s|u) &=& R(t,s) - \int_0^{u} K(t,x)K(s,x)\, \d v(x), \\
\hat P_t^G(\d x | u) &=& \Phi(\d x; \hat m_t^G(u), \hat R(t,t|u)).
\end{eqnarray*}
\end{pro}	

\begin{proof}
It is well-known that conditional Gaussian processes are Gaussian.  Therefore it is enough to identify the conditional mean and conditional covariance.  

We consider first the conditional mean.  Now
\begin{eqnarray*}
\hat m^G_t(u) &=& \E\left[G_t\big| \F_u^G\right] \\
&=&
\E\left[\int_0^t K(t,s)\, \d M_s\big|\F_u^M\right] \\
&=&
\int_0^u K(t,s)\, \d M_s \\
&=&
\int_0^u K(u,s)\, \d M_s + \int_0^u \left[K(t,s)-K(u,s)\right]\, \d M_s \\
&=&
G_u - \int_0^u (\K^{*})^{-1}\left[K(t,\cdot)-K(u,\cdot)\right](s)\, \d G_s.
\end{eqnarray*}

Let us then consider the conditional covariance. Now
\begin{eqnarray*}
\hat R_G(t,s|u) &=&
\E\left[\left(G_t-\hat m^G_t(u)\right)\left(G_s-\hat m^G_s(u)\right)\big| \F^G_u\right] \\
&=&
\E\left[\left(\int_0^t K(t,x)\, \d M_x-\hat m^G_t(u)\right)\left(\int_0^s K(s,x)\,\d M_x-\hat m^G_s(u)\right)\bigg| \F^M_u\right] \\
&=&
\E\left[\int_u^t K(t,x)\,\d M_x\int_u^s K(s,x)\,\d M_x\bigg| \F^M_u\right] \\
&=&
\E\left[\int_u^t K(t,x)\,\d M_x\int_u^s K(s,x)\,\d M_x\right] \\
&=&
\int_u^{t\wedge s} K(t,x)K(s,x)\, \d v(x) \\
&=&
\int_0^{t\wedge s} K(t,x)K(s,x)\, \d v(x) - \int_0^u K(t,x)K(s,x)\, \d v(x) \\
&=&
R(t,s) - \int_0^u K(t,x)K(s,x)\, \d v(x).
\end{eqnarray*}
\end{proof}

\section{Gaussian Volterra process with jumps}\label{sect:gvp_jumps}

In this section we prove our main result, the prediction formula for Gaussian Volterra processes with compound Poisson jumps.  Indeed, we consider the process $X=(X_t)_{t\in[0,T]}$ given by
\begin{equation}\label{eq:gvpj}
X = G + J,
\end{equation}
where $G$ is a continuous Gaussian Volterra process and $J$ is an independent compound Poisson process with intensity $\lambda$ and jump distribution $F$. In other words 
$$
J_t = \sum_{k=1}^{N_t} \xi_k,
$$ 
where $N=(N_t)_{t\in[0,T]}$ is a Poisson process with intensity $\lambda$ and the jumps $\xi_k$, $k\in\N$, are i.i.d. with common distribution $F$, and they are independent of the Poisson process $N$ and the Gaussian Volterra process $G$.  We denote
\begin{eqnarray*}
\mu_1 &=& \E[\xi_k], \\
\mu_2 &=& \E[\xi_k^2].
\end{eqnarray*}

We note that it is crucial that the Gaussian Volterra part is continuous since it implies that $\mathbb{F}^X = \mathbb{F}^{G,J}$, i.e. the signals $G$ and $J$ can be separated from the signal $X$.  We refer to \cite{Azmoodeh-Sottinen-Viitasaari-Yazigi-2014} and references therein on the continuity of Gaussian processes.

Our main theorem is the following.  

\begin{thm}[Mixed Prediction]\label{thm:predjumps}
Let $X$ be given by \eqref{eq:gvpj}. Then  
\begin{eqnarray*}
\hat m^X_t(u) 
&=& X_u - \int_0^u \Psi(t,s|u)\, \d G_s  + \lambda(t-u)\mu_1, \\
\hat R_X(t,s|u) 
&=& R(t,s) - \int_0^u K(t,x)K(s,x)\, \d v(x) + \lambda(t\wedge s -u)\mu_2, \\
\hat P^X_t(\d x|u)
&=&
\int_{y\in\R} \Phi\left(\d x -y;\hat m^G_t(u), \hat R_G(t,t|u)\right)\sum_{n=0}^\infty 
\frac{\e^{-\lambda(t-u)}(\lambda(t-u))^n}{n!}
F^{* n}\left(\d y -J_u\right).
\end{eqnarray*}
Here $F^{*n}$ is the $n$-fold convolution of the distribution $F$:
\begin{eqnarray*}
F^{*1}(\d x) &=& F(\d x), \\
F^{*n}(\d x) &=& \int_{y\in\R} F(\d x-y)F^{*(n-1)}(\d y).
\end{eqnarray*}
\end{thm}	

\begin{proof}
Let us begin with the mean $\hat m^X_t(u)$.  The conditional mean of $G$ is already known.  As for the conditional mean of $J$, we have, by independence, that
\begin{eqnarray*}
\hat m^J_t(u) &=& 
\E\left[ J_t\big| \F^{J}_u\right] \\
&=&
J_u + \E\left[ J_t-J_u\big| \F^{J}_u\right] \\
&=&
J_u + \E\left[ J_{t-u}\right] \\
&=&
J_u + \lambda(t-u)\mu_1.
\end{eqnarray*}
The formula for the conditional mean follows from this.

Let us then consider the conditional variance $\hat R_X(t,s|u)$.  By independence we have
$$
\hat R_X(t,s|u) = \hat R_G(t,s|u) + \hat R_J(t,s|u).
$$
Now $\hat R_G(t,s|u)$ is known and for $\hat R_J(t,s|u)$ we have
\begin{eqnarray*}
\hat R_J(t,s|u) &=&\Cov\left[J_t,J_s\big| \F^J_u\right] \\
&=&
\Cov\left[J_t-J_u,J_s-J_u\big|\F^J_u\right] \\
&=&
\Cov\left[J_{t-u}, J_{s-u}\right] \\
&=&
\lambda(t\wedge s-u)\mu_2.
\end{eqnarray*}

Finally, let us consider the conditional law $\hat P^X_t(\d x |u)$. By the law of total probability and independence we have
\begin{eqnarray*}
\hat P_t(\d x|u)
 &=&
\int_{y\in\R} \P\left[G_t \in \d x -y\big| J_t=y, \F_u^X\right]\P\left[J_t\in \d y \big| \F_u^X\right] \\
&=&
\int_{y\in\R} \P\left[G_t \in \d x -y\big| \F_u^G\right]\P\left[J_t\in \d y \big| \F_u^J\right] \\
&=&
\int_{y\in\R} \hat P^G_t(\d x -y|u)\P\left[J_t-J_u\in \d y -J_u\big| J_u\right]
\end{eqnarray*}
and
\begin{eqnarray*}
\lefteqn{\P\left[J_t-J_u\in \d y -J_u\big| J_u\right]} \\
&=&
\sum_{n=0}^\infty\P\left[J_t-J_u\in \d y -J_u\big| N_t-N_u=n, J_u \right]\P\left[N_{t}-N_u=n\big| J_u\right] \\
&=&
\sum_{n=0}^\infty\P\left[\sum_{k=1}^n \xi_k\in \d y -J_u\Bigg| J_u\right]\P[N_{t-u}=n] \\
&=&
\sum_{n=0}^\infty F^{* n}\left(\d y -J_u\right)\P[N_{t-u}=n]
\end{eqnarray*}
The formula for the conditional law follows from this by plugging in $\hat P^G_t(\d x|u)$ and the Poisson probabilities.
\end{proof}

\begin{rem}
It is interesting to note that just like in the Gaussian case, also in the mixed case, the conditional covariance is deterministic.
\end{rem}

\section{Examples}\label{sect:examples}

In this section we give examples for different Gaussian Volterra processes $G$ for the prediction formula of Theorem \ref{thm:predjumps}.  This means that we give the kernel $K$, the function $v$, and the kernel $\Psi$.

\begin{exa}[fBm]\label{exa:fbm}
The fractional Brownian motion $B^H$ with Hurst index $H\in(0,1)$ is a centered Gaussian process with covariance
$$
R_H(t,s) = \frac12\left(t^{2H}+s^{2H}-|t-s|^{2H}\right).
$$
By Norros et al. \cite{Norros-Valkeila-Virtamo-1999} (see also \cite{Sottinen-Viitasaari-2017b})) $B^H$ is a Gaussian Volterra process with $v(t)=t$ and the Volterra kernel
\begin{equation}\label{K_H}
K_H(t,s) = c_H\Bigg\{\left(\frac{t}{s}\right)^{H-\frac12} (t-s)^{H-\frac12}\, 
-\, (H-\frac12)s^{\frac12-H}\int_s^tu^{H-\frac{3}{2}}(u-s)^{H-\frac12}\,\d u \Bigg\},
\end{equation}	
with a normalizing constant
$$
c_H = \sqrt{\frac{2H\Gamma(\frac{3}{2}-H)}{\Gamma(\frac12+H)\Gamma(2-2H)}},
$$
and we have
\begin{equation}\label{Psi_H}
\Psi_H(t,s|u) = \frac{\sin(\pi(H-\frac{1}{2}))}{\pi} s^{\frac{1}{2}-H}(u-s)^{\frac{1}{2}-H} \int_u^t \frac{z^{H-\frac{1}{2}}(z-u)^{H-\frac{1}{2}}}{z-s}\,\d z.
\end{equation}
\end{exa}

\begin{exa}[ccmfBm]\label{exa:ccmfbm}
The long-range dependent completely correlated  mixed fractional Brownian motion (ccmfBm) was introduced in \cite{Dufitinema-Shokrollahi-Sottinen-Viitasaari-2021}. It is a Gaussian process 
$$
M = a W +  bB^H
$$
where $B^H$ is a fractional Brownian motion with $H>1/2$, and it is constructed from the Brownnian motion $W$ by using the Volterra representation
$$
B^H_t = \int_0^t K_H(t,s)\, \d W_s
$$
(see Example \ref{exa:fbm} above).  The ccfBm, $M$ is also a Gaussian Volterra process with $v(t)=t$ and the Volterra kernel
$$
K_{a,b,H}(t,s) = a\,\1_t(s) + b\,K_H(t,s).
$$
In other word, it has the following Volterra representation
$$
M_t = \int_0^t K_{a,b,H}(t,s)\, \d W_s,
$$
and so
\begin{equation*}
	\Psi_{a,b,H}(t,s|u) =
	(\mathrm{K}_{a,b,H}^*)^{-1}\left[K_{a,b,H}(t,\cdot)-K_{a,b,H}(u,\cdot)\right](s).
\end{equation*}
Here
\begin{gather*}
\mathrm{K}_{a,b,H}^*[f](t) = af(t) + \frac{b\, c(H)(H-\frac12)}{t^{H-\frac12}}\int_t^T f(u)\frac{u^{H-\frac12}}{(u-t)^{\frac{3}{2}-H}}\,\d u,
\end{gather*}
for all $f\in\Lambda_{a,b,H}$, the space of all integrands from $M$, which is simply $L^2[0,T]$ in this case. The 
$(\mathrm{K}_{a,b,H}^*)^{-1}$ is the inverse operator of $\mathrm{K}_{a,b,H}^*$, where from \cite{Dufitinema-Shokrollahi-Sottinen-Viitasaari-2021}
\begin{gather*}
	(\mathrm{K}_{a,b,H}^*)^{-1}[f](t) = f(T)K_{a,b,H}^{-1}(t,T) - \int_t^T f(s)K_{a,b,H}^{-1}(\d s, t),\\
	K_{a,b,H}^{-1}(t,s) = \frac{1}{a}\1_t(s) + \frac{1}{a}\sum_{k=1}^\infty (-1)^k\left(\frac{b}{a}\right)^k\gamma_k(t,s),\\
	\gamma_k(t,s) = \frac{c(H)^k\Gamma(H+\frac12)^k}{\Gamma(k(H-\frac12))}\frac{1}{s^{H-\frac12}}\int_s^tu^{H-\frac12}(u-s)^{k(H-\frac12)-1}\d u.
\end{gather*}
\end{exa}		 

\begin{exa}[mfBm]	
The mixed fractional Brownian motion (mfBm) is the process 
$$
\tilde M = W + B^H
$$
where the Brownian motion $W$ and the fractional Brownian motion $B^H$ are independent.  The mfBm was introduced by Cheridito \cite{Cheridito-2001}.  In Cai et al. \cite{Cai-Chigansky-Kleptsyna-2016} it was shown that the mfBm is a Gaussian Volterra process with $v(t)=t$ and a certain kernel $\tilde K_{H}$.

Indeed, let $H>1/2$ and let $L(t,s)$ be the solution of the equation
$$
L(t,s) + H(2H-1)\int_0^t L(t,x)|s-x|^{2H-2}\,\d x = -H(2H-1)|t-s|^{2H-2},\quad 0\leq s,t
$$
Then by Theorem 2.2 of \cite{Cai-Chigansky-Kleptsyna-2016} we have the following: denote
$$
\phi(t) = 1- \int_0^t L(t,x)\, \d x.
$$
Then
$$
\tilde{W}_t = \E\left[\int_0^t \phi(s)\,\d W_s\,\bigg|\, \F^{\tilde M}_t\right] = \int_0^t q(t,s)\,\d\tilde{M}_s,
$$
is a Brownian motion, where $q(t,s)$ is the unique solution of the Wiener–Hopf equation:
\begin{equation*}
	q(t,s) + H(2H-1)\int_0^tq(t,x)|s-x|^{2H-2}\,\d x = \phi(s),\quad 0\leq s,t
\end{equation*}
and 
\begin{equation*}
	\tilde{M}_t = \int_0^t\tilde{K}_H(t,s)\,\d\tilde{W}_s,
\end{equation*}
where
\begin{equation*}
	\tilde{K}_H(t,s) = -\frac{\partial}{\partial s}\int_s^t q(t,x)\d x.
\end{equation*}
and here we have
\begin{equation*}
	\tilde{\Psi}(t,s|u) =
 (\widetilde{\mathrm{K}}_H^*)^{-1}\left[\tilde{K}_H(t,\cdot)-\tilde{K}_H(u,\cdot)\right](s),
\end{equation*}
where
$$
\widetilde{\mathrm{K}}_H^*[f](t) = f(t)\tilde{K}_H(T,t) - \int_t^T\left[f(u)-f(t)\right]\tilde{K}_H(\d u, t),
$$
for all $f\in\tilde\Lambda$, the space of all integrands from $\tilde{M}$, and 
$(\widetilde{\mathrm{K}}_H^*)^{-1}$ is the inverse operator of $\widetilde{\mathrm{K}}_H^*$.
\end{exa}	


\bibliographystyle{siam}
\bibliography{pipliateekki}
\end{document}